\documentclass[a4paper,11pt]{amsart}
\usepackage{amssymb,amscd}
\usepackage{mathrsfs} 
\usepackage{tikz-cd}
\usepackage{enumitem}
\usepackage[capitalize]{cleveref}
\newcounter{alph}

\newtheorem{theo}[alph]{Theorem}

\numberwithin{equation}{section}
\newtheorem{cor}[equation]{Corollary}
\newtheorem{lem}[equation]{Lemma}

\newtheorem{prop}[equation]{Proposition}
\newtheorem{thm}[equation]{Theorem}
\theoremstyle{definition}
\newtheorem{asu}[equation]{Assumption}

\def\C{\mathbb C}
\def\F{\mathbb F}

\def\R{\mathbb R}

\def\T{\mathcal T}
\def\ve{\varepsilon}
\def\vf{\varphi}
\def\la{\langle}
\def\ra{\rangle}

\newcommand{\area}{\operatorname{area}}

\newcommand{\dt}{\operatorname{dt}}

\newcommand{\dx}{\operatorname{dx}}

\newcommand{\ess}{\operatorname{ess}}

\newcommand{\inj}{\operatorname{inj}}

\newcommand{\Lip}{\operatorname{Lip}}

\newcommand{\rk}{\operatorname{rk}}
\newcommand{\Ray}{\operatorname{Ray}}

\newcommand{\SO}{\operatorname{SO}}

\newcommand{\supp}{\operatorname{supp}}

\newcommand{\vol}{\operatorname{vol}}

\begin{document}


\title[Small eigenvalues]
{Small eigenvalues of Schr\"odinger operators over geometrically finite manifolds}
\author{Werner Ballmann}
\address
{WB: Max Planck Institute for Mathematics,
Vivatsgasse 7, 53111 Bonn}
\email{hwbllmnn\@@mpim-bonn.mpg.de}
\author{Panagiotis Polymerakis}
\address{PP: Max Planck Institute for Mathematics,
Vivatsgasse 7, 53111 Bonn}
\email{polymerp\@@mpim-bonn.mpg.de}

\thanks{\emph{Acknowledgments.}
We are grateful to the Max Planck Institute for Mathematics
and the Hausdorff Center for Mathematics in Bonn for their support and hospitality.}

\date{\today}


\begin{abstract}
We estimate the number of small eigenvalues of Schr\"odinger operators
on Riemannian vector bundles over geometrically finite manifolds.
\end{abstract}

\maketitle

\tableofcontents

\section*{Introduction}
\label{secint}

Estimating the number of small eigenvalues of complete Riemannian manifolds
is a traditional subject of study in differential geometry.
The origins lie in the theory of hyperbolic surfaces,
where eigenvalues of the Laplacian on functions are called small, if they are below $1/4$,
the bottom of the spectrum of the hyperbolic plane.
Similarly, an eigenvalue of the Laplacian on functions on real hyperbolic manifolds
is called small if it is below $(m-1)^2/4$, where $m$ denotes the dimension of the manifold.
Small eigenvalues of hyperbolic manifolds are connected to other geometric invariants,
for example the length spectrum;
see \cite{Buser92} for more information.
 
Let $M$ be a geometrically finite manifold of dimension $m$
and sectional curvature bounded by $-1\le K \le-a^2 < 0$, where $0<a<1$.
Recall that one of the four equivalent definitions of geometrically finite of Bowditch requires
that the volume of the set $U_r(C)$ of points of distance less than $r$
to the convex core $C$ of $M$ is finite for some or any $r>0$ \cite[Section 5]{Bowditch95}.
By \cite[Corollary 1.4]{BelegradekKapovitch06},
$M$ is diffeomorphic to the interior of a compact manifold with boundary (possibly empty).
Conversely, complete and connected Riemannian surfaces of pinched negative curvature
are geometrically finite if they are diffeomorphic to the interior of a compact surface
with boundary (possibly empty).
Complete and connected Riemannian manifolds of finite volume
and pinched negative sectional curvature are geometrically finite
and among the geometrically finite manifolds characterized by the property that $C=M$ or,
equivalently, that $U_r(C)=M$.

Let $H$ be the universal covering space of $M$, endowed with the lifted metric,
and denote by $\lambda_0=\lambda_0(H)$ the bottom of the spectrum of the Laplacian on functions on $H$.
We have the sharp bounds
\begin{align}\label{lambda0}
	\frac{(m-1)^2a^2}{4} \le \lambda_0 \le \frac{(m-1)^2}{4}.
\end{align}
The left inequality is due to McKean \cite{McKean70}.
The right inequality is an immediate consequence of Cheng's \cite[Theorem 1.1]{Cheng75} (as observed in \cite[Theorem 3.1]{Donnelly81b}).
In this article,
we say that an eigenvalue $\lambda$ of (the Laplacian $\Delta_0$ on functions on) $M$ is \emph{small}
if $\lambda\le\lambda_0$.
We are interested in estimating the number of small eigenvalues of $M$,
where we count eigenvalues always with multiplicity.

For any $\lambda\ge0$, denote by $N(\lambda)$ the dimension of the image
of the spectral projection of $\Delta_0$ with respect to $[0,\lambda]$.
Then $N(\lambda)$ is equal to the finite number of eigenvalues of $M$ in $[0,\lambda]$
for $\lambda<\lambda_{\ess}(M)$ and $N(\lambda)=\infty$ for $\lambda>\lambda_{\ess}(M)$,
where $\lambda_{\ess}(M)$ denotes the bottom of the essential spectrum of $\Delta_0$.

\begin{theo}\label{main}
For any $0<\ve\le\lambda_0$,
\begin{align*}
	\frac{N(\lambda_0-\ve)}{\vol(U_1(C))} \le C(m,a,\ve) < \infty.
\end{align*}
In particular, $\lambda_{\ess}(M)\ge\lambda_0$.
\end{theo}

The main point of \cref{main} is of course that the constant $C(m,a,\ve)$
does not depend on the specific $H$ or its quotient $M$,
but only on $m$, $a$, and $\ve$, as indicated by the notation.
By what we said above, $U_1(C)$ is to be read as $M$
in the case where the volume of $M$ is finite.

A complete Riemannian metric on a surface $S$ of finite type
and negative Euler characteristic $\chi(S)$ has at most $-\chi(S)$ eigenvalues in $[0,\lambda_0]$,
see \cite[Th\'eor\`emes 1 and 2]{OtalRosas09} and, for the result in the stated generality,
\cite[Theorem 1.5]{BallmannMatthiesenMondal17}.
Together with the Gau\ss{}-Bonnet formula and the lower curvature bound $-1$,
it gives that $N(\lambda_0)\area(S)^{-1}\le1/2\pi$ in the case where $\area(S)<\infty$.
This is a sharp estimate.
In the general case of \cref{main}, sharp estimates are not known.

In the case where $M$ is compact or, more generally, of finite volume,
Buser, Colbois, and Dodziuk showed that the number of eigenvalues of $M$
in $[0,(m-1)a^2/4]$ is at most $C(m,a)\vol(M)$ \cite[Theorem 3.6]{BuserColboisDodziuk93}.
It would be interesting, in that case, to get their kind of bound,
but for the number of eigenvalues of $M$ in $[0,\lambda_0]$.
The main--and only--issue here is an extension of their estimate of the second Neumann eigenvalue
of domains which they call pieces of cheese \cite[(3.11)]{BuserColboisDodziuk93}
(see also \cite[pp.\,61--64]{Buser80}).

Hamenst\"adt considered the case of geometrically finite manifolds of infinite volume.
She obtained that $N((m-1)a^2/4-\ve)e^{-\vol(U_1(C))}$
is bounded by a constant $C(m,a,\ve)$ \cite[Theorem.2]{Hamenstaedt04}.
In the hyperbolic case, that is,
in the case where $M$ is a quotient of one of the hyperbolic spaces $H=H_\F^\ell$,
she obtained the same kind of bound,
but with $\lambda_0$ in place of $(m-1)a^2/4$ \cite[Corollary.2]{Hamenstaedt04}.
In contrast to her estimates, our estimate is linear in the volume of $U_1(C)$.

In her recent article \cite{Hamenstaedt19},
Hamenst\"adt considered small eigenvalues also for geometrically finite manifolds of finite volume,
but her estimates are of a somewhat different nature.
Namely, she obtains lower estimates in terms of Neumann eigenvalues of specific large domains in $M$.
Corollaries \ref{hamfin} and \ref{haminf} are versions of her result for both cases,
$M$ of finite or infinite volume,
but for Dirichlet eigenvalues of sufficiently large domains in $M$.

For geometrically finite hyperbolic manifolds $M$, Li shows that $N(\lambda_0)$ is finite, that is,
that the number of eigenvalues of $M$ in $[0,\lambda_0]$ is finite \cite[Theorem 1.1]{Li20}. 
(He only asserts this for the half-open interval $[0,\lambda_0)$.)
However, his arguments do not seem to give an estimate on the number of all small eigenvalues.

With some care, our arguments also work in the case of Schr\"odinger operators
on vector bundles $E$ over $M$, whose pull-back to $H$ is a bundle associated
to a geometric structure, with structure group $G$ a covering group of $\SO(m)$,
via an orthogonal representation $\theta$ of $G$ on a Euclidean space $E_0$.
Such bundles inherit a Riemannian metric and a metric connection $\nabla$.
Examples are twisted versions of tensor bundles, spinor bundles, flat bundles,
and bundles obtained from them by natural operations.

Let $E$ be as above and $A=\Delta+V$ be a formally self-adjoint Schr\"odinger operator on $E$,
where $\Delta=\nabla^*\nabla$ denotes the connection Laplacian.
The pull-backs of $E$ and $A$ to the universal covering space $H$ of $M$
will be denoted by $E_H$ and $A_H$.
Assume that the potential $V$ is bounded from below.
Then both, $A$ and $A_H$, are bounded from below and, in particular,
essentially self-adjoint \cite[Theorem A.24]{BallmannPolymerakis20a}.

Denote by $\lambda_0=\lambda_0(A_H)=\lambda_0(A,H)$ the bottom of the spectrum of $A_H$.
We say that an eigenvalue $\lambda$ of $A$ is \emph{small} if $\lambda\le\lambda_0$.
For any $\lambda\in\R$, denote by $N_A(\lambda)$ the dimension of the image
of the spectral projection of $A$ with respect to $(-\infty,\lambda]$.
Then $N_A(\lambda)$ is equal to the finite number of eigenvalues of $A$ in $(-\infty,\lambda]$
for $\lambda<\lambda_{\ess}(A)$ and $N_A(\lambda)=\infty$ for $\lambda>\lambda_{\ess}(A)$,
where $\lambda_{\ess}(A)$ denotes the bottom of the essential spectrum of $A$.

\begin{theo}\label{maina}
For any $\ve>0$,
\begin{align*}
	\frac{N_A(\lambda_0-\ve)}{\vol(U_1(C))} \le C(m,a,|\theta_*|,\ve)\rk(E).
\end{align*}
In particular, $\lambda_{\ess}(A)\ge\lambda_0$.
\end{theo}

\cref{main} is an instance of the case $\theta_*=0$ of \cref{maina}.
The second assertion of \cref{maina} is a special case of \cite[Theorem B]{BallmannPolymerakis22a}.

\subsection{Applications to differential forms}\label{forms}
The most interesting situation of \cref{maina} is when there is a natural lower bound for $A$
which is strictly smaller than $\lambda_0$.
For example, if
\begin{align*}
	A = \Delta_k = (d+d^*)^2
\end{align*}
is the Hodge-Laplacian on differential $k$-forms, then zero is a natural lower bound.
On the other hand, if $\delta_k=(m-1-k)a-k>0$ for some $0\le k<(m-1)/2$, then
\begin{align}\label{lambdak}
	\lambda_0(\Delta_{k},H) = \lambda_0(\Delta_{m-k},H) \ge \delta_k^2/4 > 0,
\end{align}
by \cite[Theorem 2]{Kasue94} (or by the later \cite[Corollary 5.4]{BallmannBruening01}),
and this estimate is sharp \cite[Example 5.5]{BallmannBruening01}.
For the real hyperbolic space with sectional curvature $-1$, that is, $a=1$, \eqref{lambdak} is an equality.
In fact, for all the hyperbolic spaces $H=H_\F^\ell$, except possibly for the octonionic hyperbolic plane, 
$\lambda_0(\Delta_{k},H)$ has been determined explicitly \cite{Donnelly81,Pedon99,Pedon05}.
If we normalize their metric so that the minimum of their sectional curvature is $-1$ and if $k<m/2$, then
\begin{align}
	\lambda_0(\Delta_{k},H) = \lambda_0(\Delta_{m-k},H)
	= \begin{cases}
	(m-1-2k)^2/4 &\text{if $H=H_\R^m$}, \\
	(\ell-k)^2/4 &\text{if $H=H_\C^\ell$}. \\ 
	\end{cases}
\end{align}
Moreover, for all hyperbolic spaces, $\lambda_0(\Delta_{k},H_\F^\ell)>0$ unless $|k-m/2|<1$,
and then $\lambda_0(\Delta_{k},H_\F^\ell)=0$.
In the latter case, our estimate does not apply.

For $0\le k\le m$,
denote by $\mathcal H^k(M)$ the space of square-integrable harmonic $k$-forms on $M$,
that is, the eigenspace of $\Delta_k$ for the eigenvalue zero.
In the context of \cref{maina}, it is interesting to recall that,
if $M$ is a geometrically finite real hyperbolic manifold of even dimension $m=2k$,
$\mathcal H^k(M)$ is infinite dimensional if the volume of $M$ is infinite,
by \cite[Theorem 1.5]{MazzeoPhillips90} of Mazzeo and Phillips.
Since $\lambda_0(\Delta_k,H^{2k}_\R)=0$,
this is an example for the possibility that $N_A(\lambda_0)=\infty$,
contrasting the finiteness result of Li in the case of $\Delta_0$,
which we cited above.

In \cite{DiCerboStern20}, Di Cerbo and Stern obtain upper bounds for the dimensions
of the spaces $\mathcal H^k(M)$
on compact or finite volume real and complex hyperbolic manifolds $M$.
For example,
in the case of compact complex hyperbolic manifolds $M=\Gamma\backslash H_\C^\ell$
and $k<\ell=m/2$, they obtain that
\begin{align*}
	\frac{\dim\mathcal H^k(M)}{\vol M} \le \frac{C(m,k)}{\vol B(r)^{1-k/m}},
\end{align*}
where $r=\inj(M)$ is the injectivity radius of $M$ and $B(r)\subseteq H_\C^{\ell}$
is a geodesic ball of radius $r$ \cite[Theorem 2]{DiCerboStern20}.
In comparison, \cref{maina} implies that
\begin{align*}
	\frac{\dim\mathcal H^k(M)}{\vol M}
	\le C(m,1/2,k,(\ell-k)^2/8)\binom{m}{k} = C'(m,k).
\end{align*}
Their estimate is better if $\inj(M)$ is close to infinity.
Di Cerbo and Stern have similar estimates in other cases,
but we refer to their article for further comparison, discussion, and references.

\subsection{Some remarks before we start}
\label{susome}
The proofs of Theorems \ref{main} and \ref{maina} were inspired by the articles
\cite{BuserColboisDodziuk93} of Buser, Colbois, and Dodziuk
and \cite{Hamenstaedt04} by Hamenst\"adt.
Our way of estimating the number of small eigenvalues is via Neumann eigenvalues
with respect to coverings of certain domains in $M$ by small geodesic balls,
and to that end we need that non-zero Neumann eigenvalues of $A$
on sufficiently small geodesic balls in $M$ are arbitrary large.
It is not difficult to see this in the case of the Laplacian on functions.
In the general case, where we need control on the connection of $E$,
this holds true in the case where the pull-back of $E$ to the universal covering of $M$
is an associated bundle as above.
We will see this in the last section of the text, where we discuss both cases separately.

The assertion in \cref{maina} is stable under adding a constant to the potential $V$ of $A$.
For that reason, we assume throughout the body of the text that $V\ge0$.
In particular, we then have that $A\ge0$.

\section{Preparations}
\label{secpre}
Throughout the article,
we let $H$ be a simply connected and complete Riemannian manifold
with negative sectional curvature $-1\le K\le-a^2$.
We let $\Gamma$ be a group acting properly discontinuously and isometrically on $H$
such that any element of $\Gamma$, that fixes some point of $H$, acts by the identity on $H$.
Then the quotient $M=\Gamma\backslash H$ is a manifold
and $H\to M$ is the universal covering projection.

\subsection{Estimating covering multiplicities}
\label{submul}
Given $\sigma>0$, we say that $Z\subseteq M$ is $\sigma$-separated
if $d(z,z')\ge\sigma$ for all $z,z'\in Z$.

\begin{lem}\label{lebesgue}
Let $B\subseteq M$ and assume that $\inj(x)\ge\sigma>0$ for all $x\in B$.
Let $Z\subseteq B$ be a maximal $2\sigma$-separated subset.
Then the balls $B_\sigma(z)$, $z\in Z$, are pairwise disjoint,
and the balls $B_{2\sigma}(z)$ cover $B$.
Moreover, if $s\ge2$ and $0<(2s+1)\sigma\le\ln2$,
then any $x\in M$ is contained in at most $\ell_s$ of the balls $B_{s\sigma}(z)$,
where
\begin{align*}
	\ell_s = \ell_s(m) = (2s+1)^m5^{m-1}/4^{m-1}.
\end{align*}
\end{lem}

\begin{proof}
The first assertions are clear.
As for the last,
if $y\in B_{s\sigma}(z)\cap B_{s\sigma}(z')$ with $z,z'\in Z$, then $d(z,z')<2s\sigma$.
Hence, for any such $z,z'$, we have $B_{\sigma}(z')\subseteq B_{(2s+1)\sigma}(z)$.
Since balls of radius $\sigma$ about points in $Z$ are pairwise disjoint,
volume comparison gives that
\begin{align*}
	\ell =
	a^{m-1}\int_0^{(2s+1)\sigma}\sinh(t)^{m-1}\dt\bigg/\int_0^{\sigma}\sinh(at)^{m-1}\dt
\end{align*}
is an upper bound for the number of such $z'$ (including $z$).
Note here that the Bishop-Gromov volume bound applies to the balls $B_{(2s+1)\sigma}(z)$,
although the radius might be beyond $\inj(z)$.
The rough estimate \[t\le\sinh(t)\le 5t/4\] for $0\le t\le\ln2$ gives $\ell\le\ell_s$.
\end{proof}

\subsection{Margulis lemma and thick--thin decomposition}
\label{submar}

We recall and refine the setup in \cite[Section 4]{BallmannPolymerakis22a}.
For any subgroup $\Gamma'$ of $\Gamma$, $\rho>0$, and $x\in H$,
let $\Gamma'_\rho(x)$ be the subgroup of $\Gamma'$
generated by the elements $g\in\Gamma'$ with $d(x,gx)<\rho$.
We call
\begin{align}\label{thix}
	\T_\rho(\Gamma') = \{ x\in H \mid \text{$\Gamma'_\rho(x)$ is non-trivial}\}
	\hspace{3mm}\text{and}\hspace{3mm}
	H \setminus \T_\rho(\Gamma')
\end{align}
the \emph{$\rho$-thin} and \emph{$\rho$-thick part of $H$} with respect to $\Gamma'$ and
\begin{align}\label{thio}
	\T_\rho(M') = \Gamma'\backslash\T_\rho(\Gamma')
	\hspace{3mm}\text{and}\hspace{3mm}
	M' \setminus \T_\rho(M')
\end{align}
the \emph{$\rho$-thin} and \emph{$\rho$-thick part of $M'=\Gamma'\backslash H$}, respectively.
Clearly, the $\rho$-thick part of $M'$ is the set of $x\in M'$ with injectivity radius $\inj(x)\ge\rho/2$.

By the Margulis lemma, there is an explicit constant $\rho_0=\rho_0(m)>0$
such that $\Gamma'_\rho(x)$ is virtually nilpotent if $0<\rho\le\rho_0(m)$
\cite[§ 9.5]{BallmannGromovSchroeder85}.
Recall that $\rho_0(m)$ is a very small positive number so that, by far, $10\rho_0(m) < \ln2$.
For $0<\rho\le\rho_0$, there are two sources for the components of $\T_{\rho}(M)$,
parabolic ends and short closed geodesics;
compare with \cite[§10]{BallmannGromovSchroeder85}.

Let $P\subseteq H_\iota$ be the set of parabolic points of $\Gamma$.
For any $p\in P$, the $\rho$-thin part $U_p=\T_{\rho}(\Gamma_p)$ of $H$
with respect to the  isotropy group $\Gamma_p$ of $p$ in $\Gamma$
is a component of $\T_{\rho}(\Gamma)$.
Note that $U_p\cap U_q=\emptyset$ for all $p\ne q$ in $P$.
Clearly, $U_{gp}=gU_p$ for all $g\in\Gamma$,
and therefore the image $V_p$ of $U_p$ in $M$ only depends on the orbit of $p$ under $\Gamma$.
The components of the ${\rho}$-thin part of $M$ due to parabolic ends is the union
\begin{align*}
	V_{\rm par} = \cup V_p,
\end{align*}
which is disjoint up to passing to $\Gamma$-orbits.

Next, let $Q_{\rho}$ be the set of geodesics $[x,y]$ with $x\ne y$ in $H_\iota$
such that $[x,y]$ is the axis of a hyperbolic isometry $h\in\Gamma$
which shifts $[x,y]$ by less than ${\rho}$.
Then the isotropy group $\Gamma_{[x,y]}$ of $[x,y]$ (as a set) in $\Gamma$
is infinite cyclic.
For any geodesic $[x,y]$ in $Q_{\rho}$,
the ${\rho}$-thin part $U_{[x,y]}=T_{\rho}(H,\Gamma_{[x,y]})$ of $H$ with respect to $\Gamma_{[x,y]}$
is a component of $T_{\rho}(\Gamma)$.
Note that the different $U_{[x,y]}$ are pairwise disjoint and that they are also disjoint
from the $U_p$ above.
Clearly, $U_{[gx,gy]}=gU_{[x,y]}$ for all $g\in\Gamma$,
and therefore the image $V_{[x,y]}$ of $U_{[x,y]}$ in $M$ only depends on the orbit of $[x,y]$ under $\Gamma$.
The components of the ${\rho}$-thin part of $M$ due to short closed geodesics is the union
\begin{align*}
	V_{\rm scc} = \cup V_{[x,y]},
\end{align*}
which is disjoint up to passing to $\Gamma$-orbits.

Note that the above $V_p$ and $V_{[x,y]}$ depend on the choice of $\rho$,
hence also $V_{\rm par}$ and $V_{\rm scc}$.

\subsection{Harmonic coordinates}
\label{subharm}
Using harmonic coordinate gives us the following;
see \cite[Section 5]{JostKarcher82} (or also \cite[Main Lemma 2.2]{Anderson90}).

\begin{lem}\label{harmonic}
Given $\delta$, there is a $\sigma=\sigma(m,\delta)>0$ such that, for each $x\in H$,
there is a coordinate system \[U_x\to B_\sigma^m(0)\subseteq\R^m\] about $x$ mapping $x$ to $0$,
such that
\begin{align*}
	g_0\le g \le (1+\delta)^2g_0
	\quad\text{and}\quad
	|dg| \le \delta,
\end{align*}
where $g_0$ denotes the Euclidean metric on $B_\sigma^m(0)$.
\end{lem}

Note that the injectivity radius does not interfere here since it is infinite for all points in $H$.
Note also that, with respect to coordinates about $x$ as above,
\begin{align}\label{3balls}
	B_{\rho}(x) \subseteq B_\rho^m \subseteq B_{(1+\delta)\rho}(x)
\end{align}
for all $0<\rho\le\sigma(m,\delta)$,
where the ball in the middle is the Euclidean ball with center at the origin.
For each $x\in H$, we fix such coordinates and let $U_\rho(x)$ be the neighborhood of $x$
which corresponds to $B_\rho^m$.

Note that the Christoffel symbols associated to the coordinates are bounded in terms of $\delta$.
Apply the Gram-Schmidt orthonormalization process to the coordinate frame $(\partial/\partial x^i)$
to get an orthonormal frame $X=(X_i)$ on $U_x$.
Since the Gram-Schmidt process only involves the Riemannian metric $g$,
\begin{align}\label{hframe}
	|X_i-\partial/\partial x^i|\le c_0(m)\delta
	\quad\text{and}\quad
	|\nabla_{X_i}X_j|\le c_0(m)\delta.
\end{align}

\subsection{Associated bundles}
\label{susab}
View an orthonormal frame of $H$ at $x\in H$ in two ways,
namely as a basis of $T_xH$ or as an isomorphism $\R^m\to T_xH$,
and similarly for a Riemannian vector bundle $E$ over $H$.
A frame of $H$ or $E$ over an open subset $U$ of $H$
is a family of frames of $H$ or $E$ which is defined and smooth on $U$.

Fix an orientation of $H$ and let $\SO(H)$ be the principal bundle
of oriented orthonormal frames of $H$ with structure group $\SO(m)$.
Let $G\to\SO(m)$ be a covering of connected groups
and $P\to M$ be a principal bundle with structure group $G$
together with a commutative triangle
\begin{equation}\label{abct}
\begin{tikzcd}
	P \arrow[r] \arrow[dr] & \SO(H) \arrow[d] \\
	& H
\end{tikzcd}
\end{equation}
such that the horizontal arrow is compatible with the right actions of $G$ on $P$
and $\SO(m)$ on $\SO(H)$, respectively.
The main examples are the identity of $\SO(H)$ and spin structures,
but, in dimension two, there are further possiblilities.

Let $\theta$ be an orthogonal representation of $G$ on a Euclidean vector space $E_0$
and let $E\to H$ be the bundle over $H$ associated to $\theta$.
Recall that $E$ consists of equivalence classes $[p,u]$ with $p\in P$ and $u\in E_0$,
where $[pg,u]=[p,gu]$ for all $p\in G$.
We also write $p_\theta u=[p,u]$ and consider $p_\theta$ as a frame of $E$.
Fix an orthonormal basis $(f_1,\dots,f_k)$ of $E_0$.

The scalar product of $E_0$ induces a Riemannian metric on $E$.
Furthermore, the Levi-Civita connection induces a metric connection on $E$
(as does any other metric connection on $M$) by the rule
\begin{align}\label{abcd}
	\nabla_X(p_\theta u)
	= p_\theta(du(X) + \theta_*(\gamma(X))u),
\end{align}
where we write a smooth section of $E$ locally in the form $p_\theta u$
for some local frame $p_\theta$ of $E$ and a smooth map $u$ to $E_0$
and where $\gamma$ is the $(2,1)$-tensor field of Christoffel symbols
of the frame of $H$ associated to $p$.

\subsection{An upper bound for $\lambda_0$}
\label{subupp}
An immediate consequence of \cite[Theorem 1.1]{Cheng75} is that,
for $A=\Delta_0$, the Laplacian on functions, we have
\begin{align}\label{lambda3}
	\lambda_0(\Delta_{0,H}) \le \lambda_{\ess}(\Delta_{0,H}) \le \frac{(m-1)^2}4.
\end{align}
Namely, Cheng shows that the eigenfunctions for the smallest Dirichlet eigenvalue
on balls in the model space, when transferred to balls in $H$ of the same radius,
have Rayleigh quotient at most the one in the model space.

Whereas \eqref{lambda3} is an optimal upper bound,
we do not know an optimal upper bound for $\lambda_0$ in the general case.

\begin{prop}\label{lambdae}
For $E_H$ an associated bundle as in \cref{susab} and $\Delta_H$ its connection Laplacian,
we have
\begin{align*}
	 \lambda_0(\Delta_H)\le \lambda_{\ess}(\Delta_H) \le c_1(m,|\theta_*|).
\end{align*}
\end{prop}

\begin{proof}
Let $\sigma=\sigma(m,1/2)$ in \cref{harmonic} and $x\in H$.
Fix harmonic coordinates on $U_x$ about $x$
and consider the orthonormal frame $(X_i)$ on $U_x$ as in \cref{subharm}.
Let $p$ be a lift of the frame to $P$ over $U_x$.
Then sections of $E$ over $U_x$ can be written as $[p,u]=p_\theta u$,
where $u$ is a map from $U_x$ to $E_0$.
Choose a unit vector $u_0\in E_0$ and let $u(y)=\vf(|y|/\sigma)u_0$,
where $\vf(t)=1-t$ on $[0,1]$ and $U_x$ is viewed as the Euclidean ball $B_\sigma^m(0)$.
Then
\begin{align*}
	\|\nabla_{X_i}(p_\theta u)\|_2
	&= \|p_\theta(d\vf(X_i)u_0 + \vf\theta_*(\gamma(X_i))u_0)\|_2 \\
	&\le (1 + |\theta_*|c_0(m,1/2))\vol U_x
	\le c_1(m,|\theta_*|) \|u\|_2. 
\end{align*}
Since $x\in H$ is arbitrary, this yields the right inequality.
\end{proof}

\subsection{Amenable coverings}
\label{subame}
In the proof of our main results we will use \cite[Proposition 4.13]{Polymerakis20a},
which we summarize as follows:

\begin{prop}\label{pp}
Let $p\colon M_2\to M_1$ be an amenable Riemannian covering of Riemannian manifolds.
Let $A_1$ be a formally self-adjoint differential operator on a vector bundle $E_1$ over $M_1$
and $A_2$ be the lift of $A_1$ to the lift $E_2$ of $E_1$.
Then, for any non-zero $u_1\in C^\infty_c(M_1,E_1)$, $\lambda\in\R$, and $\ve>0$,
there exists a non-zero $u_2\in C^\infty_c(M_2,E_2)$ with $\|u_2\|_2 = \|u_1\|_2$ such that
\begin{align*}
	\Ray_{A_2}(u_2) \le \Ray_{A_1}(u_1) + \ve
\end{align*}
\end{prop}

The proof consists of a sophisticated choice of cutoff functions
to turn the lift of $u_1$ to $M_2$ into a section with the asserted properties.
Amenability makes such choices possible.

\section{On Dirichlet eigenvalues}
\label{secdir}

We refine the setup in \cite[Section 4]{BallmannPolymerakis22a}.
Let $C$ be the convex core of $M$ and $\pi_C\colon M\to C$ be the projection.
Let $\rho_0=\rho_0(m,b)$ be the Margulis constant as in \cref{submar} and
\begin{align}\label{cthick}
	K_0 = C \setminus \T_{\rho_0}(M)
\end{align}
be the $\rho_0$-thick part of $C$, that is, the set of points $x\in C$ with $\inj(x)\ge\rho_0/2$.
By the definition of geometrically finite, $K_0$ is compact.

\begin{lem}\label{finvol}
If $|M|<\infty$ and $u\in C^\infty_c(M\setminus K_0,E)$,
then \[\la Au,u\ra_2\ge\lambda_0\|u\|_2^2.\]
\end{lem}

\begin{proof}
Since $|M|<\infty$, we have $C=M$ and therefore $M\setminus K_0=\T_{\rho_0}(M)$.
Hence we can assume that the support of $u$ is contained in one of the $V_p$ or $V_{[x,y]}$
as described in \cref{submar}.
Now the fundamental group of any of the $V_p$ or $V_{[x,y]}$ is infinite and amenable
and its universal covering is an open subset of $H$.
Hence the assertion follows from \cref{pp}.
\end{proof}

Our next aim is to get a similar statement in the case $|M|=\infty$.
To that end, we let, in this section,
\begin{align}\label{fixrho1}
	\rho = \rho_0(m)/2.
\end{align}
Let $R$ be a maximal $2\rho$-separated set of points in $K_0$.
Then the balls $B_{\rho}(x)$ about points $x\in R$ are pairwise disjoint,
and the balls $B_{2\rho}(x)$ cover $K_0$.
We obtain an open covering $\mathcal V$ of $C$ by open sets $V_x$,
where $x\in S=P\cup Q\cup R$ and $V_x=B_{3\rho}(x)$ for $x\in R$.

\begin{lem}\label{lebesgue1}
Each $y\in M$ is contained in at most $\ell_3(m)+1$ open sets $V_x$ from $\mathcal V$.
\end{lem}

\begin{proof}
Since the sets $V_x$ with $x\in P\cup Q$ are pairwise disjoint,
$y$ is contained in at most one of these $V_x$
(up to passing to $\Gamma$-orbits).
On the other hand, since $\inj(x)\ge\rho$ for all $x \in K_0$,
$y$ is contained in at most $\ell_3(m)$ of the balls $B_{3\rho}(x)$ with $x\in R$,
by \cref{lebesgue}.
\end{proof}

For any $x\in S$, let $\tau_x$ be the continuous function on $C$
that is equal to one on points of $C\cap V_x$ which are away from $C\setminus V_x$ by at least $\rho$,
vanishes on $C\setminus V_x$,
and is linear in the distance to $C\setminus V_x$ in between.
Each $\tau_x$ is Lipschitz continuous with Lipschitz constant $1/\rho$.

\begin{lem}\label{taux}
For any $y\in C$,
\begin{align*}
	\sum_{x\in S} \tau_x(y)^2 \ge 1/2.
\end{align*}
\end{lem}

\begin{proof}
This is clear for $y\in K_0$
since the balls $B_{2\rho}(x)$ with $x\in R$ already cover $K_0$.
It is also clear for $y\in V_x$, where $x\in P\cup Q$,
if $y$ has distance at least $\rho$ to $C\setminus V_x$, since then $\tau_x(y)=1$.
In the remaining case, if $y$ has distance at most $\rho$ to $C\setminus V_x$,
consider a shortest geodesic connection $\gamma$ from $y$ to $C\setminus V_x$.
The endpoint $y'$ of $\gamma$ belongs to $K_0$, by the definition of $K_0$
and since the various $V_z$ with $z\in P\cup Q$ are pairwise disjoint.
But then $y'\in B_{2\rho}(z)$ for some $z\in R$,
and hence $\tau_x(y)+\tau_z(y)\ge1$.
\end{proof}

Define functions $\psi_x$, $x\in S$, by
\begin{align}\label{psix}
	\psi_x = \frac{\tau_x}{\left(\sum_{z\in S}\tau_z^2\right)^{1/2}}.
\end{align}
Then the $\psi_x$ are a family of functions,
whose squares form a partition of unity on $C$.
They admit Lipschitz constant
\begin{align}\label{psixl}
	\Lip(m,\rho) =  2^{5/2}(\ell_3(m)+1)/\rho
\end{align}
by Lemmas \ref{lebesgue1} and \ref{taux} and the quotient rule for Lipschitz functions.

For any $x\in S$, let $\vf_x=\psi_x\circ\pi_C$.
Now the $V_x$ with $x\in R$ are contractible,
and therefore we are in the situation considered in \cite[Section 5]{BallmannPolymerakis22a}.
Hence we can conclude from  \cite[(5.7)]{BallmannPolymerakis22a} that,
for any non-zero $u\in C^\infty_c(M,E)$, there is an $x\in S$ such that $\vf_xu$ is non-zero and
\begin{align}\label{5dot7}
	\Ray_A(\vf_xu)
	\le \Ray_A(u) + \sum_{x\in S} \|\nabla\vf_x\|_{\supp u,\infty}^2.
\end{align}
This is a keystone of our discussion.

Denote by $r$ the distance to $K_0=\{r=0\}$ and recall that $\lambda_0=\lambda_0(A_H)$.
Our analogue of \cref{finvol} in the general case is as follows.

\begin{lem}\label{repsil}
For any $\ve>0$, there is an $r_\ve=r_\ve(m,a)\ge1$ such that
\begin{align*}
	\la Au,u\ra_2
	\ge (\lambda_0-\ve)\|u\|_2^2
\end{align*}
for all $u\in C^\infty_c(\{r\ge r_\ve\},E)$.
In particular, $\lambda_{\ess}(A)\ge\lambda_0$.
\end{lem}

We have to be a bit more careful than in \cite{BallmannPolymerakis22a}
since we need to estimate the error term in \eqref{5dot7},
that is, the Lipschitz constants of the $\vf_x$.

\begin{proof}
By \cref{finvol}, we may assume that $|M|=\infty$.
This is only for convenience, since the following discussion would be a little bit clumsier
if the finite volume case would be included.

Since $|M|=\infty$, we have $C\ne M$.
Since $C$ is convex, the projection $\pi_C$ onto $C$ has Lipschitz constant $\le1/\cosh(as)$
on the set $\{d_C\ge s\}$, where $d_C$ denotes the distance to $C$.
Therefore $\Lip\vf_x\le\Lip(m,\rho)/\cosh(as)$ on $\{d_C\ge s\}$.
By \cref{lebesgue1}, at any point of $M$,
at most $\ell_3(m)+1$ points contribute to the error term in \eqref{5dot7}.
Therefore the error term at any point of $\{d_C\ge s\}$ is bounded by
\begin{align}\label{error}
	2^5(\ell_3(m)+1)^3/\rho^2\cosh(as)^2.
\end{align}
Now there are the following three kinds of points $y$ in $\{d_C\ge s\}$: \\
1) If $\pi_C(y)\in K_0$, then $r(y)=d_C(y)$,
and we have \eqref{error} as an upper estimate for the error term. \\
2)  If $\pi_C(y)\in C\cap V_x$ for some $x\in P\cup Q$
and $\pi_C(y)$ is away from $C\setminus V_x$ by at most distance $3\rho$,
then $d_C(y)\ge r(y)-3\rho$. \\
3) If $\pi_C(y)\in C\cap V_x$ for some $x\in P\cup Q$
and $\pi_C(y)$ is away from $C\setminus V_x$ by at least distance $3\rho$,
then $\psi_z(\pi_C(y))=0$ for all $z\ne x$ in $S$ and $\psi_x=1$ in a neighborhood of $\pi_C(y)$.
But then the error term in \eqref{5dot7} vanishes.

Recall now that $\rho=\rho_0/2$ and let $r_\ve= r_\ve(m,a,b)>0$ be the solution of
\begin{align}\label{error2}
	\cosh(a(r_\ve-3\rho_0/2))^2= 2^5(\ell_3(m)+1)^3/\rho_0^2\ve.
\end{align}
We claim that $r_\ve$ satisfies the assertion of \cref{repsil}.
To that end, let $u\in C^\infty_c(\{r\ge r_\ve\},E)$ be non-zero.
If the asserted inequality would not hold for $u$,
then there would be an $x\in S$ such that $\Ray_A(\vf_xu)<\lambda_0$,
by \eqref{5dot7} and the definition of $r_\ve$.
But the support of $\vf_xu$ is contained in $V_x$,
hence $\vf_xu$ can be lifted to $H$ for $x\in R$,
and then $\Ray_A(\vf_xu)\ge\lambda_0$.
Or else the support of $\vf_xu$ belongs to $V_x$ for some $x\in P\cup Q$.
Now the fundamental group of $V_x$ is amenable and its universal covering is an open subset of $H$.
Then again $\Ray_A(\vf_xu)\ge\lambda_0$, by \cref{pp}.
Hence both cases, $x\in R$ or $x\in P\cup Q$, lead to a contradiction.
\end{proof}

Given $\ve>0$, let $r_\ve=0$ if $|M|<\infty$ and as in \cref{repsil} otherwise.
Suppose that $u\in C^\infty_c(M,E)$ has Rayleigh quotient $<\lambda_0-\ve$. 
Let
\begin{align}\label{sepsil}
	s_\ve = r_\ve+\pi/\sqrt{2\ve}
\end{align}
and $\vf$ be the function such that
\begin{align}\label{fepsil}
	\vf(x) =
	\begin{cases}
	1 &\text{on $\{r\le r_\ve\}$}, \\
	\cos(\sqrt\ve(r-r_\ve)/\sqrt2) &\text{on $\{r_\ve\le r\le s_\ve\}$}, \\
	0  &\text{on $\{r\ge s_\ve\}$}.
	\end{cases}
\end{align}

\begin{thm}\label{thmdir}
Let $\ve>0$, $0\le\lambda<\lambda_0-2\ve$,
and $W\subseteq C^\infty_c(M,E)$ be a subspace
such that $\Ray_A(u)\le\lambda$ for all non-zero $u\in W$.
Then multiplication with $\vf$ defines a linear injection of $W$ into $\Lip_c(\{r\le s_\ve\},E)$ such that
\begin{align*}
	\Ray_A(\vf u)\le\lambda+\ve<\lambda_0-\ve
\end{align*}
for all non-zero $u\in W$.
In particular, $A$ over $\{r\le s_\ve\}$
has at least $\dim W$ Dirichlet eigenvalues in $[0,\lambda+\ve]$.
\end{thm}

\begin{proof}
Set $\psi=\sqrt{1-\vf^2}$.
Then the pair $(\vf^2,\psi^2)$ is a partition of unity on $M$.
By \cref{finvol} respectively \cref{repsil}, we have
\begin{align*}
	\la A(\psi u),\psi u\ra_2 \ge (\lambda_0-\ve) \|\psi u\|_2^2
\end{align*}
for any $u\in W$.
Therefore we must have, by \eqref{5dot7}, that $\vf u$ is non-zero with
\begin{align*}
	\Ray_A(\vf u)
	\le \Ray_A(u) + \|\nabla\vf\|_\infty^2 + \|\nabla\psi\|_\infty^2
	\le \lambda + \ve
	< \lambda_0 - \ve.
\end{align*}
Hence multiplication by $\vf$ is an injective linear map from $W$ to a subspace of
Lipschitz sections with compact support of $E$ contained in $\{r\le s_\ve\}$
and with Rayleigh quotient $\le\lambda+\ve$.
\end{proof}

Let $0<2\ve<\lambda_0$, and enumerate the eigenvalues of $A$ below $\lambda_0-2\ve$ as 
\begin{align*}
  0\le\lambda_0(A,M)\le\dots\le\lambda_k(A,M)<\lambda_0-2\ve.
\end{align*}
Similarly, enumerate the Dirichlet eigenvalues of $A$ over $\{r\le s_\ve\}$ as
\begin{align*}
  0\le\lambda_0^D(A,\{r\le s_\ve\})\le\lambda_1^D(A,\{r\le s_\ve\})\le\dots
\end{align*}
where, in each case, multiplicities are taken into account.

\begin{cor}\label{cordir}
For all $0<2\ve<\lambda_0$ and $1\le i\le k$, we have
\begin{align*}
  \lambda_i(A,M) \ge \lambda_i^D(A,\{r\le s_\ve\})-\ve.
\end{align*}
\end{cor}

In the case of the Laplacian $\Delta_0$ on functions,
we use the shorthand $\lambda_i(M)$ instead of $\lambda_i(\Delta_0,M)$,
and similarly for the Dirichlet eigenvalues of  $\{r\le s_\ve\}$.

\begin{cor}\label{hamfin}
If the volume of $M$ is finite and $m\ge3$, then
\begin{align*}
	\lambda_i^D(\{r\le s_\ve\})
	\le \big(1+\frac{\ve|M|^2}{C(m)a^2}\big)\lambda_i(M)
\end{align*}
for all $1\le i\le k$.
\end{cor}

\begin{proof}
Since $m\ge3$ and the volume of $M$ is finite,
the bottom of the positive part of the spectrum of $M$ satisfies
\begin{align*}
  \lambda_1(M) \ge C(m)a^2/|M|^2,
\end{align*}
by Dodziuk's \cite[Theorem]{Dodziuk87}.
Therefore, if $i\ge1$, then
\begin{align*}
	\lambda_i^D(\{r\le s_\ve\})
	\le \lambda_i(M)+\ve
	\le \big(1+\frac{\ve|M|^2}{C(m)a^2}\big)\lambda_i(M)
\end{align*}
where we use \cref{cordir} and that $\lambda_i(M)\ge\lambda_1(M)$.
\end{proof}

\cref{hamfin} is a version of Hamenst\"adt's \cite[Theorem 1]{Hamenstaedt19},
where she obtains that $\lambda_i(M)\ge\lambda_i^N(\T_\rho(M))/3$
for all eigenvalues $\lambda_i(M)\le(m-2)^2a^2/12$,
where $\rho>0$ is sufficiently small and the superindex $N$ indicates Neumann eigenvalues.
(She also assumes that $M$ is orientable.)

\begin{cor}\label{haminf}
If the volume of $M$ is infinite and $m\ge3$, then
\begin{align*}
	\lambda_i^D(\{r\le s_\ve\})
	\le \big(1+\frac{\ve|U_1(M)|^2}{C(m,a)}\big)\lambda_i(M)
\end{align*}
for all $0\le i\le k$.
\end{cor}

\begin{proof}
By Hamenst\"adt's \cite[Theorem.1]{Hamenstaedt04}, we have
\begin{align*}
  \lambda_0(M) \ge C(m,a)/|U_1(M)|^2
\end{align*}
for the smallest eigenvalue $\lambda_0(M)$ of $M$.
The assertion follows now as in the previous case.
\end{proof}

\section{Counting small eigenvalues}
\label{seccou}

Recall that we normalized $A=\Delta+V$ so that $V\ge0$
and that we use the shorthand $\lambda_0=\lambda_0(A_H)$.
Recall also that $r$ denotes the distance to the $\rho_0$-thick part $K_0$
of the convex core $C$ as in \eqref{cthick}.

In this section, we estimate the number of small eigenvalues under

\begin{asu}\label{asu}
There is a $0<\rho_2=\rho_2(m,a,|\theta_*|)\le\rho_0/2$ such that, for all $x\in H$ and $0<\rho\le\rho_2$,
there is an open neighborhood $U_\rho(x)$ with
\begin{align*}
	B_{\rho}(x) \subseteq U_\rho(x) \subseteq B_{3\rho/2}(x)
\end{align*}
such that the $(\rk E+1)^{\rm st}$ Neumann eigenvalue of $\Delta=\nabla^*\nabla$ on $U_\rho(x)$
is at least $\ell_3(m)\lambda_0$.
\end{asu}

In \cref{secneu}, we will provide conditions under which \cref{asu} holds.

\begin{lem} \label{inrest}
For any $x\in M$, $\inj(x) \ge e^{-r(x)}\rho_0/2$.
\end{lem}

\begin{proof}
Since $C$ is convex, the projection $\pi_C$ onto $C$ has Lipschitz constant one.
Hence $\inj(x)\ge\inj(\pi_C(x))$, for any $x\in M$.

On the other hand, since $\angle_{\pi_C(x)}(x,y)\ge\pi/2$ for any $y\in K_0$,
we have $r(x)\ge r(\pi_C(x))$.
Hence, by the first step, it suffices to estimate $\inj(z)$ for any $z\in C\setminus K_0\subseteq C$.
But then $z$ belongs to the $\rho_0$-thin part of $C$.
More precisely, there is a parabolic point $p\in H_\iota$ such that $x\in V_p$
or a geodesic $[x,y]$ in $H$, which is shifted by an $h\in\Gamma$ by less than $\rho_0$,
such that $x\in V_{[x,y]}$.
Then $\partial U\cap K_0\ne\emptyset$, where $U=V_p$ or $U=V_{[x,y]}$, respectively,
and hence
\begin{align*}
	2\inj(x) \ge e^{-d(x,\partial V)}\rho_0
	\ge e^{-d(x,K_0)}\rho_0
	= e^{-r(x)}\rho_0
\end{align*}
as asserted.
\end{proof}

Let $\omega(m,a,r)$ be the volume of the ball of radius $r$ in the real hyperbolic space
of dimension $m$ and constant sectional curvature $-a^2$.
Recall that $\omega(m,a,r)$ grows like $\exp((m-1)ar)$ as $r$ tends to infinity.

\begin{lem}\label{numest}
For any $s\ge1$ and $0<\rho\le e^{-s}\rho_0/2$,
the cardinality of any $2\rho$-separated set $Z$ in $\{r\le s\}$ is bounded by
\begin{align*}
	\frac{|\{r\le s+1\}|}{\omega(m,a,\rho)}.
\end{align*}
\end{lem}

\begin{proof}
For any $2\rho$-separated set $Z$ in $\{r\le s\}$,
the balls of radius $\rho$ about the points of $Z$ are pairwise disjoint and contained in $\{r\le s+1\}$,
where we use the crude estimate $\rho_0\le2e^{s}$ for the latter assertion.
Moreover, since $\inj(z)\ge e^{-s}\rho_0/2\ge\rho$ for any $z\in Z$,
the volumes of these balls is at least $\omega(m,a,\rho)$.
\end{proof}

\begin{thm}\label{thmnum}
For any $\ve>0$, $A$ over $\{r\le s_\ve\}$ has at most
\begin{align*}
	\frac{|\{r\le s_\ve+1\}|}{\omega(m,a,e^{-s_\ve}\rho_2/3)}\rk E
\end{align*}
Dirichlet eigenvalues in $[-\infty,\lambda_0)$.
\end{thm}

\begin{proof}
Let $3\rho/2 = e^{-s_\ve}\rho_2 \le e^{-s_\ve}\rho_0/2$.
Let $Z$ be a maximal $\rho$-separated subset of $\{r\le s_\ve\}$.
Then the balls $B_{\rho}(z)$, with $z\in Z$, cover $\{r\le s_\ve\}$
and $\inj(z)\ge3\rho/2$ at each $z\in Z$, by \cref{inrest}.
Hence the sets $U_\rho(z)$ are well defined as images of the leaves in $H$ above them.
They cover $\{r\le s_\ve\}$ and any $x\in M$ is contained in at most $\ell_3(m)$ of them.

Suppose that $W\subseteq C^\infty(\{r\le s\},E)$ is a subspace with
\begin{align*}
	\dim W = \frac{|\{r\le s_\ve+1\}|}{\omega(m,a,e^{-s_\ve}\rho_2/3)}\rk E +1.
\end{align*}
Then there is a $u\in W$ such that $u$ is perpendicular to the first $\rk E$ Neumann eigensections
on each of the sets $U=U_\rho(z)$.
But then $\Ray_A(u|_U)\ge\ell_3(m)\lambda_0$ on each of these sets.
In conclusion,
\begin{align*}
	\la Au,u\ra_2 \ge \ell_3(m)^{-1} \sum_U \la Au,u\ra_{U,2}
	\ge \sum_U \lambda_0 \|u\|_{U,2}^2 
	\ge \lambda_0 \|u\|_2^2.
\end{align*}
Now the upper bound for the number of Dirichlet eigenvalues below $\lambda_0$
follows from the variational characterization of eigenvalues.
\end{proof}

To finish the proof of Theorems \ref{main} and \ref{maina},
we need to compare the volume of $U_1(C)$ to that of $\{r\le s_\ve+1\}$.
In fact, we have

\begin{lem}\label{volest}
For all $s\ge1$,
\begin{align*}
	|\{r\le s\}| \le |U_s(C)| \le \frac{\omega(m,1,s+2)}{\omega(m,1,1)}|U_1(C)|.
\end{align*}
\end{lem}

\begin{proof}
Choose a maximal $2$-separated subset $Z$ of $C$. It is evident that the balls $B(z,1) \subset U_1(C)$,
$z \in Z$, are pairwise disjoint, and the balls $B(z,s+2)$, $z \in Z$, cover $U_s(C)$ for any $s \geq 1$.
We derive from the Bishop-Gromov comparison theorem that
	\begin{align*}
		| \{r \leq s\} & \leq |U_s(C)| \leq \sum_{z \in Z} |B(z,s+2)|\\
		&\leq \frac{\omega(m,1,s+2)}{\omega(m,1,1)} \sum_{z \in Z} |B(z,1)|
		\leq \frac{\omega(m,1,s+2)}{\omega(m,1,1)}|U_1(C)|,
	\end{align*}
as we wished.
\end{proof}

\begin{proof}[Proof of Theorems \ref{main} and \ref{maina} under \cref{asu}]
Given $\ve>0$, we get from Theorems \ref{thmdir} and \ref{thmnum} and \cref{volest}
that the number $N_A(\lambda_0-\ve)$ of eigenvalues of $A$ below $\lambda_0-\ve$ is at most
\begin{align*}
	\frac{|\{r\le s_\ve+1\}|}{\omega(m,a,e^{-s_\ve}\rho_2/3)}\rk E
	\le \frac{\omega(m,1,s_\ve+3)}{\omega(m,1,1)\omega(m,a,e^{-s_\ve}\rho_2/3)}\rk E|U_1(C)|
\end{align*}
Since $s_\ve=r_\ve+\pi/\sqrt{2\ve}$ and $r_\ve$ only depends on $m$ and $a$,
the assertions of Theorems \ref{main} and \ref{maina} follow.
\end{proof}

\section{On Neumann eigenvalues}
\label{secneu}

We discuss two cases, in which \cref{asu} holds.
The discussion is based on approximating $\Delta$ on geodesic ballls
or approximate geodesic balls in $H$ by the Euclidean Laplacian on Euclidean balls.
To that end,
let $\lambda_N(m)>0$ be the first non-zero Neumann eigenvalue of the Euclidean Laplacian
on functions on the Euclidean unit ball $B_1^m\subseteq\R^m$.
Then the first non-zero Neumann eigenvalue of the Euclidean Laplacian on functions
on the Euclidean ball $B_\rho^m\subseteq\R^m$ of radius $\rho$ is
\begin{align}\label{eucneu}
	\lambda_N(m,\rho) = \lambda_N(m)\rho^{-2}
\end{align}
since the Laplacian behaves like $\text{\rm length}^{-2}$ and $B_\rho^m=\rho B_1^m$.

\subsection{Laplacian on functions}
\label{subfun}
We use Riemannian normal coordinates to view the ball $B_\rho(x)$
as $B_\rho^m$ endowed with a Riemannian metric $g$
and to compare the latter with the Euclidean metric on $B_\rho^m$.
Indicating Euclidean objects by an index $0$, we have
\begin{align*}
	g_0 \le g \le \frac{\sinh(\sigma)^2}{\sigma^2}g_0,
\end{align*}
by the Rauch comparison theorem.
Hence, if $\sinh(\rho)^2/\rho^2\le1+\delta$, then
\begin{align*}
	(1+\delta)^{-1}|df|_0^2 \le |df|^2 \le |df|_0^2
\end{align*}
for the differential of $f\in C^\infty(B_\rho(x))$ and
\begin{align*}
		\vol_0 \le \vol \le (1+\delta)^{m/2}\vol_0
\end{align*}
for the volume elements.
For the Rayleigh quotient of a non-zero $f\in C^\infty(B_\rho(x))$, we obtain
\begin{align}\label{eucray}
	(1+\delta)^{-m/2-1}\Ray_0 f \le \Ray f \le (1+\delta)^{m/2}\Ray_0 f.
\end{align}
Together with \eqref{eucneu}, this shows that \cref{asu} holds for the Laplacian on functions.
More precisely, we get that any two-dimensional subspace of $C^\infty(B_\rho(x))$
contains a non-zero function $f$ such that
\begin{align*}
	\Ray f \ge (1+\delta)^{-(m+2)/2}\lambda_N(m)\rho^{-2}
\end{align*}
so that \cref{asu} holds for $U_\rho(x)=B_\rho(x)$ and $0<\rho\le\rho_1$ as long as
\begin{align*}
	\ell_3(m)\lambda_0\rho^2 \le (1+\delta)^{-(m+2)/2}\lambda_N(m).
\end{align*}
Since the latter condition only depends on $m$, we get that $\rho_2=\rho_2(m,a)$,
where we observe that $\theta_*=0$ in this case anyway.

Clearly, the case where $E$ is flat is analogous and with the same result,
except that the subspace of sections of $E$ over $B_\rho(x)$ should be of dimension $\rk E+1$
instead of two.

\subsection{Connection Laplacian on sections}
\label{subbun}
In \cref{subfun}, the connection does not enter the discussion.
This changes in the case of general $E$, where we need control over the connection.
This is the reason why we consider associated bundles.
Assume therefore that $E_H$ is associated to a $G$-structure $P\to H$
via an orthogonal representation $\theta$ of $G$ on a Euclidean space $E_0$;
compare with \cref{susab}.
Let $\delta>0$ and $U_\sigma(x)$ be the neighborhood of $x$ with harmonic coordinates
as in \cref{harmonic}.

\begin{lem}
In the above situation, there is a constant $C_0=C_0(m)$ such that
the Rayleigh quotients of $u$ on $U_\rho(x)$ with respect to the Euclidean metric
and the Riemannian metric satisfy
\begin{align*}
	\Ray(u)
	\ge (1-\delta)^{m+3} \Ray_0 u - m\delta(1-\delta)C_0^2|\theta|_*^2.
\end{align*}
\end{lem}

\begin{proof}
Consider the orthonormal frame $(X_i)$ as in \cref{subharm}.
If $A=(a_i^j)$ is an orthogonal matrix
and $(Y_i)$ is the orthonormal frame with $Y_i=a_i^jX_j$, then also
\begin{align*}
	|\nabla_{Y_i}Y_j| \le c_0(m)\delta.
\end{align*}
Let now $y\in U_x$.
Then there is an orthogonal matrix $A$ such that the vectors $(Y_i(y))$
are pairwise orthogonal with respect to the Euclidean metric on $U_x$.
Their Euclidean norm is between $1-\delta$ and $1$.
Let $(F_\mu)$ be the orthonormal frame of $E$ over $U_x$ associated to $(Y_i)$
and write sections of $E$ over $U_x$ as linear combinations $u=u^\mu F_\mu$.
Then we get, at the point $y$,
\begin{align*}
	|\nabla u|^2
	&\ge \sum_{i}
	 |\{Y_i(u)+\theta_*(\gamma(Y_i))u|^2 \\
	 &\ge \sum_{i} \{
	 |Y_i(u)|^2 - 2|Y_i(u)||\theta_*(\gamma(Y_i))u| + |\theta_*(\gamma(Y_i))u|^2\} \\
	 &\ge \sum_{i} \{
	 (1-\delta)|Y_i(u)|^2 - \frac{1-\delta}\delta|\theta_*(\gamma(Y_i))u|^2\} \\
	 &\ge \sum_{i} \{
	 (1-\delta)^3|\partial_iu|^2 - \frac{1-\delta}\delta|\theta_*|^2c_0(m)^2\delta^2|u|^2\} \\
	 &\ge (1-\delta)^3|\nabla_0u|^2 - m\delta(1-\delta)c_0(m)^2|\theta|_*^2|u|^2,
\end{align*}
where $\gamma$ is the tensor field of type $(2,1)$ defined by $\gamma(Y_i)Y_j=\gamma_{ij}^kY_k$.
In conclusion, we obtain
\begin{align*}
	\Ray u
	&= \frac{\int|\nabla u|^2\dx}{\int|u|^2\dx} \\
	&\ge (1-\delta)^3 \frac{\int|\nabla_0 u|^2\dx}{\int|u|^2\dx}
	- m\delta(1-\delta)c_0(m)^2|\theta|_*^2 \\
	&\ge \frac{(1-\delta)^{3}}{(1+\delta)^{m}}\frac{\int|\nabla_0 u|^2\dx_0}{\int|u|^2\dx_0}
	- m\delta(1-\delta)c_0(m)^2|\theta|_*^2 \\
	&\ge (1-\delta)^{m+3} \Ray_0 u - m\delta(1-\delta)c_0(m)^2|\theta|_*^2.
\end{align*}
for the Rayleigh quotients of any non-zero section $u$ of $E$ over $U_x$.
\end{proof}

With $U_\rho(x)=B_\rho^m$,
we obtain that any subspace of $C^\infty(B_\rho^m,E)$ of dimension $\rk E+1$
contains a non-zero section $u$ such that
\begin{align*}
	\Ray u 
	\ge (1-\delta)^{m+3} \lambda_N(m)\rho^{-2} - m\delta(1-\delta)c_0(m)^2|\theta|_*^2
\end{align*}
so that \cref{asu} holds for all $0<\rho\le\rho_1=\rho_1(m,a)$ as long as
\begin{align*}
	(\ell_6(m)\lambda_0+m\delta(1-\delta)c_0(m)^2|\theta|_*^2)\rho^2
	\le (1-\delta)^{m+3} \lambda_N(m).
\end{align*}
Choosing $\delta=1/2$, the latter condition only depends on $m$ and $|\theta_*|$
and yields a $\rho_2=\rho_2(m,a,|\theta_*|)$ as required in \cref{asu}.

\bibliographystyle{amsplain}
\bibliography{GeFiSpec}
\end{document}